\documentclass[12pt, letterpaper]{article}
\usepackage[utf8]{inputenc}
\usepackage[margin=2.54cm]{geometry}
\usepackage{amsmath, amsthm, amssymb}
\usepackage{hyperref}
\usepackage{enumitem}
\usepackage{graphicx}
\usepackage{standalone}
\usepackage{mathrsfs}
\usepackage{mathtools}
\usepackage{verbatim}

\newtheorem{theorem}{Theorem}
\newtheorem*{theorem*}{Theorem}
\newtheorem{lemma}[theorem]{Lemma}
\newtheorem{definition}[theorem]{Definition}
\newtheorem{corollary}[theorem]{Corollary}
\newtheorem{proposition}[theorem]{Proposition}    \newtheoremstyle{TheoremNum}
{\topsep}{\topsep}
{\itshape}
{}
{\bfseries}
{.}
{ }
{\thmname{#1}\thmnote{ \bfseries #3}}
\theoremstyle{TheoremNum}

\newcommand{\mb}{\mathbb}
\newcommand{\mc}{\mathcal}
\newcommand{\say}[1]{\text{\hspace{1cm}#1}}

\author{Tolson Bell\thanks{thbell@cmu.edu. Research supported by NSF Graduate Research Fellowship grant DGE 2140739.}\\Carnegie Mellon University}
\date{May 2023}
\title{\vspace{-1cm}The Park--Pham Theorem\\with Optimal Convergence Rate}
\begin{document}
\maketitle

\begin{abstract}
Park and Pham's recent proof of the Kahn--Kalai conjecture was a major breakthrough in the field of graph and hypergraph thresholds. Their result gives an upper bound on the threshold at which a probabilistic construction has a $1-\epsilon$ chance of achieving a given monotone property. While their bound in other parameters is optimal up to constant factors for any fixed $\epsilon$, it does not have the optimal dependence on $\epsilon$ as $\epsilon\rightarrow 0$. In this short paper, we prove a version of the Park--Pham Theorem with optimal $\epsilon$-dependence.
\end{abstract}
\section{Introduction}
One of the most fundamental tasks in probabilistic combinatorics is finding the thresholds for graph and hypergraph properties. At what $p$ should you expect $G(n,p)$ to be more likely than not to contain a triangle? A Hamiltonian cycle? A common first attempt is to lower bound this $p$ by the first moment method. The Park--Pham Theorem essentially says that applying the first moment method on some structure that is necessary for your desired graph to appear is always within a logarithmic factor of the true threshold.

Let $H$ be a hypergraph on a finite vertex set $X$. The \underline{upward closure} of $H$ is \[\langle H\rangle=\{R\subseteq X:\exists~S\in H\text{ s.t. }S\subseteq R\},\]that is, the subsets of $X$ that contain a hyperedge in $H$. A hypergraph $G$ \underline{undercovers} $H$ if $H\subseteq\langle G\rangle$, that is, every hyperedge in $H$ contains a hyperedge in $G$.

\begin{definition}[\cite{KahnKalai}, Section 1]
Let $q\in(0,1)$. $H$ is \underline{$q$-small} if there is a hypergraph $G$ such that $H\subseteq\langle G\rangle$ and $\sum_{R\in G}q^{|R|}\le\frac 12$.
\end{definition}
Let \underline{$X_p$} denote a subset of $X$ where each element is included independently with probability $p$, let \underline{$p_c(H)$} be the probability $p_c$ such that $\mb{P}(\exists~S\in H$ s.t. $S\subseteq X_{p_c})=\frac 12$ (if $H$ is non-trivial, this must exist and be unique by monotonicity), let \underline{$\ell(H)$} be the size of the largest hyperedge of $H$, and let \underline{$q(H)$} be the maximum $q$ such that $H$ is $q$-small.

To apply the Park--Pham Theorem, $X$ should be the set of objects that are being selected independently at random, for example, the edges of $G(n,p)$ or hyperedges of a random hypergraph, and $H$ is the graph or hypergraph property you want to find the threshold of, for example, hyperedges consisting of minimal edge sets which form a triangle or Hamiltonian cycle. $\ell(H)$ is the maximal number of elements that make up one instance of that property, for example, $n$ for Hamiltonian cycles and 3 for triangles. $q(H)$ is often relatively easy to compute for structured $H$. Our main goal is to find $p_c(H)$. One motivation for the definition of $q$-small is that $q(H)\le p_c(H)$ by first moment method on the probability of some $R\in G$ being contained in $X_q$, a necessary condition for some $S\in H$ to be contained in $X_q$. The Park--Pham Theorem shows that $q(H)$ gives rise to an upper bound, as well as the lower bound, on $p_c(H)$.

\begin{theorem}[Kahn--Kalai Conjecture, now Park--Pham Theorem]\label{kkpp}
For any hypergraph $H$, \[p_c(H)\le 8q(H)\log(2\ell(H)).\]
\end{theorem}
In this theorem and throughout the paper, all logarithms are base 2. The Park--Pham Theorem (with arbitrary constant) was conjectured by Kahn and Kalai \cite{KahnKalai}, who called it ``extremely strong'' and showed many applications of it, and was proven by  Park and Pham \cite{ParkPham}, building off previous work \cite{ALWZ,FKNP}. Theorem \ref{kkpp} will be proven in Section 2. Our proof achieves a significantly lower constant and avoids some of the complications of the original Park--Pham proof. We use similar techniques to Rao \cite{SoilOil}.

The Park--Pham Theorem we prove is equivalent to saying that for $q>q(H)$, $\ell=\ell(H)$, and $p=8q\log(2\ell(H))$, we have $\mb{P}(\exists~S\in H$ s.t. $S\subseteq X_p)>\frac 12$. But you may want to know more than just when $G(n,p)$ has a 50/50 chance of containing a triangle or Hamiltonian cycle; you may want to know when it has a .999 chance of containing these. So a natural question is to replace $\frac 12$ with $1-\epsilon$ for any $\epsilon>0$. Park and Pham also proved an $O(q(H)\log(\ell(H)))$ upper bound for any fixed $\epsilon$, but as it was not their focus, their dependence on $\epsilon$ is exponentially worse than the dependence we give. The following theorem is our main result:
\begin{theorem}\label{MainTheorem}
Let $H$ be a hypergraph that is not $q$-small and let $\epsilon\in(0,1)$. Let $p=48q\log\left(\frac{\ell(H)}{\epsilon}\right)$. Then $\mb{P}(\exists~S\in H$ s.t.~$S\subseteq X_p)>1-\epsilon$.
\end{theorem}
This bound is optimal up to constant factors, that is, has the optimal dependence on all of $\ell$, $\epsilon$, and $q$. We will prove Theorem \ref{MainTheorem} in Section 3, and then will relate it to other work in Section 4.
\section{Proof of the Park--Pham Theorem}
Let $H$ be a \underline{$\ell$-bounded} hypergraph, that is, $|S|\le\ell$ for every $S\in H$. Given a set $W$, and $S\in H$, let \underline{$T(S,W)$} be $S'\setminus W$ for $S'=argmin_{S'\in H:S'\subseteq W\cup S}|S'\setminus W|$ (break ties arbitrarily). Note that for a given $W$, we have that $\{T(S,W):S\in H\}$ undercovers $H$, as for every $S\in H$, we have $T(S,W)\subseteq S$. If $H$ is not $q$-small, then $\{T(S,W):S\in H\}$ is also not $q$-small, as any $G$ undercovering $\{T(S,W):S\in H\}$ also undercovers $H$.
\begin{proposition}[\cite{ParkPham}, Lemma 2.1]\label{few}
Let $H$ be any $\ell$-bounded hypergraph (which may or may not be $q$-small!) and $L>1$. Let $1\le t\le\ell$ and $\mc{U}_t(H,W)=\{T(S,W):S\in H,|T(S,W)|=t\}$. Let $W$ be chosen uniformly at random from $\binom{X}{Lq|X|}$. Then \[\mb{E}_W\sum_{U\in\mc{U}_t(H,W)}q^t<L^{-t}\binom\ell t.\]
\end{proposition}
\begin{proof}[Proof of Proposition \ref{few}]
We will follow the proof of Park and Pham \cite{ParkPham}. It is equivalent for us to show that $\sum_{W\in\binom{X}{Lq|X|}}|\mc{U}_t(H,W)|<\binom{|X|}{Lq|X|}L^{-t}q^{-t}\binom{\ell}{t}$. To achieve an upper bound on the number of $T=T(S,W)$, it suffices to give a procedure for uniquely specifying any valid $(W,T)$ pair (where the $T$ is $T(S,W)$ for that $W$ and some $S\in H$).

First, fix a universal ``tiebreaker'' function $\chi:\langle H\rangle\rightarrow H$ such that $\chi(Y)\subseteq Y$ for all $Y\in\langle H\rangle$.

Now, specify $Z=W\sqcup T$. Note that these two sets are disjoint by definition, so this has size exactly $Lq|X|+t$ and we have at most \[\binom{|X|}{Lq|X|+t}=\binom{|X|}{Lq|X|}\prod_{i=1}^t\frac{|X|-Lq|X|-i+1}{Lq|X|+i}\le\binom{|X|}{Lq|X|}(Lq)^{-t}\] valid choices.

Now, we claim that $T\subseteq\chi(Z)$. We know $Z\in\langle H\rangle$ since $S\subseteq Z$, so $\chi(Z)\subseteq Z=W\sqcup T$. If $T\not\subseteq\chi(Z)$, we could not have that $T$ was the minimizer, as we could have instead taken $S'=\chi(Z)$ and then $\chi(Z)\backslash W\subseteq T$.

We can thus specify $T$ (with $|T|=t$) as a subset of $\chi(Z)$ (with $|\chi(Z)|\le\ell$ since $\chi(Z)\in H$), so there are at most $\binom\ell t$ choices for $T$.

This process specified $T$ and the disjoint union of $T$ and $W$, so we have also specified $W$, and thus have given a way to specify every possible $(W,T(S,W))$ pair, with at most $\binom{|X|}{Lq|X|}(Lq)^{-t}\binom{\ell}{t}$ possible choices.
\end{proof}
\begin{proof}[Proof of Theorem \ref{kkpp} from Proposition \ref{few}]
We will iterate the process of replacing each $S$ by $T(S,W)$. Start with $H_0=H$, $X_0=X$, and $\ell_0=\ell(H)$. We will choose $W_i$ to be a uniformly random set in $\binom{X_i}{8q|X_i|}$. Set \[\mc{C}_i=\cup_{t=\lfloor\ell_i/2\rfloor+1}^{\ell_i}\mc{U}_t(H_i,W_i)\text{~~~and~~~}H_{i+1}=\{T(S,W_i):S\in H_i,|T(S,W_i)|\le\ell_i/2\}.\]Set $\ell_{i+1}=\lfloor\ell_i/2\rfloor$. Note that $H_{i+1}$ is a $\ell_{i+1}$-bounded hypergraph on $X_{i+1}=X\setminus\cup_{j=0}^iW_j$.

Now, we repeat until we reach an $i=I$ where $\ell_{I+1}<1$, which then gives $H_{I+1}=\emptyset$ or $H_{I+1}=\{\emptyset\}$. As $\ell_i\le2^{-i}\ell$, we have $I\le\lfloor\log(\ell)\rfloor$. Let $\mc{U}=\cup_{i=0}^I\mc{C}_i$ and $W=\cup_{i=0}^IW_i$. Now, we claim that either there is some $S\in H$ such that $S\subseteq W$ (in which case $H_{I+1}=\{\emptyset\}$ and we have succeeded), or else $H_{I+1}=\emptyset$ and $\mc{U}$ undercovers $H$. This is true because if you trace any $S_0=S$ through $S_{i+1}=T(S_i,W_i)$, there is some $0\le i\le I$ at which we have either $T(S_i,W_i)\in\mc{C}_i$, which gives that this $\mc{C}_i$ undercovers $S$; or $T(S_i,W_i)=\emptyset$, which means that $S_i\subseteq W_i$, and thus there exists an $S'\in H$ that is in $\cup_{j=0}^{i-1}W_j$, as $S_i=S'\setminus\cup_{j=0}^{i-1}W_j$ for some $S'\in H$.

Therefore, to show that there is a high probability of some $S\in H$ being in $W$, it suffices to show there is a low probability of $\mc{U}$ undercovering $H$. For each $1\le t\le\ell$, let $i(t)$ be the highest $i$ such that $\ell_i\ge t$. Note that sets of size $t$ are only added to $\mc{U}$ at one step of our process, into $\mc{C}_{i(t)}$. In other words, $\mc{U}_t=\mc{U}_t(H_{i(t)},W_{i(t)})$ and $\mc{U}=\cup_{t=1}^{\ell(H)}\mc{U}_t$. Then \begin{align*}
\mb{E}\sum_{U\in\mc{U}}q^{|U|}&=\mb{E}\sum_{t=1}^\ell\sum_{U\in\mc{U}_t}q^t=\sum_{t=1}^\ell\mb{E}\sum_{U\in\mc{U}_t}q^t<\sum_{t=1}^\ell 8^{-t}\binom{\ell_{i(t)}}{t}\say{(Proposition \ref{few} with $L=8$)}\\&\le\sum_{t=1}^\ell 8^{-t}\binom{2t-1}{t}<\sum_{t=1}^48^{-t}\binom{2t-1}{t}+\sum_{t=5}^\infty 8^{-t}2^{2t-1}=\frac{819}{4096}+\frac{1}{32}<\frac 14.
\end{align*}
If $\mc{U}$ undercovers $H$, then $\sum_{U\in\mc{U}}q^{|U|}>\frac 12$. By Markov's Inequality, $\mb{P}(\sum_{U\in\mc{U}}q^{|U|}>\frac 12)<\frac 12$. Using for the first time that $H$ is $q$-small, this means that with probability more than half, $\mc{U}$ does not undercover $H$ and thus some $S\in H$ has $S\subseteq W$.

$W$ is then a uniformly random set of size $\sum_{i=0}^I8q|X_i|\le\sum_{i=0}^{\lfloor\log(\ell)\rfloor}8q|X|=p|X|$. If we make $W=X_p$ rather than a random set of size $p|X|$, Theorem \ref{kkpp} still holds: we can freely add elements to $X$ that are not in any hyperedge in $H$ and take the limit as the number of these points goes to infinity.
\end{proof}
\section{Proof of Theorem \ref{MainTheorem}}
In this section, we will prove the main theorem of this paper, which we recall below:
\begin{theorem}[\ref{MainTheorem}]
Let $H$ be a $\ell$-bounded hypergraph that is not $q$-small and let $\epsilon\in(0,1)$. Let $p=48q\log\left(\frac{\ell}{\epsilon}\right)$. Then $\mb{P}(\exists~S\in H$ s.t.~$S\subseteq X_p)>1-\epsilon$.
\end{theorem}
As a warm-up to the proof of Theorem \ref{MainTheorem}, we first note that we can quickly get logarithmic $\epsilon$-dependence if we allow a product of $\log(\ell)$ and $\log(1/\epsilon)$ instead of a sum:
\begin{proposition}\label{ppcor}
Let $H$ be a $\ell$-bounded hypergraph that is not $q$-small and let $\epsilon\in(0,1)$. Let $p=8q\log(2\ell)\lceil\log\left(\frac{1}{\epsilon}\right)\rceil$. Then $\mb{P}(\exists~S\in H$ s.t.~$S\subseteq X_p)>1-\epsilon$.
\end{proposition}
\begin{proof}
Note that if some set $W$ does not contain a hyperedge in $H$, then $H'=\{S\backslash W:S\in H\}$ undercovers $H$, and is thus is also not $q$-small. So let $H_0=H$ and $X_0=X$. For all $1\le i\le\lceil\log\left(\frac{1}{\epsilon}\right)\rceil$, we take $W_i=(X_i)_{8q\log(2\ell)}$, take $X_{i+1}=X_i\backslash W_i$, and take $H_{i+1}=\{S\backslash W_i:S\in H_i\}$. At each step, either we have some $S\in H_i$ such that $S\subseteq X_i$ or $H_{i+1}$ is $\ell$-bounded and not $q$-small. Thus, at each step, if we have not yet succeeded, we have probability $>\frac 12$ of $W_i$ containing a hyperedge in $H_i$ by Theorem \ref{kkpp}. So after $\lceil\log\left(\frac{1}{\epsilon}\right)\rceil$ steps, we have that $W=\cup_{i=1}^{\lceil\log(1/\epsilon)\rceil}W_i$ has more than a $1-\epsilon$ chance of containing some hyperedge in $H$. We have $W\sim X_{1-(1-8q\log(2\ell))^{\lceil\log(1/\epsilon)\rceil}}$ and $1-(1-8q\log(2\ell))^{\lceil\log(1/\epsilon)\rceil}<p$, so this is also true for $W\sim X_p$.
\end{proof}
However, we will see in the next section that Proposition \ref{ppcor} is not the bound we want. What the above proof does give us is the important idea that, in this setting, we can repeat a random trial where success is more likely than failure until it succeeds.
\begin{proof}[Proof of Theorem \ref{MainTheorem}]
As in the proof of Theorem \ref{kkpp}, we will iterate the process of replacing each $S$ by $T(S,W)$, starting with $H_0=H$, $X_0=X$, and $\ell_0=\ell(H)$. We choose $W_i$ uniformly at random from $\binom{X_i}{8q|X_i|}$ and let $\mc{C}_i=\cup_{t=\lfloor\ell_i/2\rfloor+1}^{\ell_i}\mc{U}_t(H_i,W_i)$. The main difference now is that we will have a ``success'' and a ``failure'' criteria at each stage. By Proposition \ref{few}, we know that \[\mb{E}\sum_{U\in\mc{C}_i}q^{|U|}=\mb{E}\sum_{t=\lfloor\ell_i/2\rfloor+1}^{\ell_i}\sum_{U\in\mc{U}_t}q^t<\sum_{t=\lfloor\ell_i/2\rfloor+1}^{\ell_i}8^{-t}\binom{\ell_i}{t}.\]
We consider step $i$ a ``failure'' if \[\sum_{U\in\mc{C}_i}q^{|U|}>2\left(\sum_{t=\lfloor\ell_i/2\rfloor+1}^{\ell_i}8^{-t}\binom{\ell_i}{t}\right).\]By Markov's inequality, success is more likely than failure at every step. We always set $X_{i+1}=X_i\setminus W_i$. If step $i$ fails, then we keep $H_{i+1}=\{S\backslash W_i:S\in H_i\}$ (or for that matter, $H_{i+1}=\{T(S,W_i):S\in H_i\}$) and $\ell_{i+1}=\ell_i$, essentially keeping the same hypergraph and retrying. If step $i$ succeeds, then as before we set $H_{i+1}=\{T(S,W_i):S\in H_i,|T(S,W_i)|\le\ell_i/2\}$ and $\ell_{i+1}=\lfloor\ell_i/2\rfloor$. In either case, $H_{i+1}$ is a $\ell_{i+1}$-bounded hypergraph on $X\setminus\cup_{j=0}^iW_j$, so our claim of success on step $i+1$ being more likely than failure still holds. If $H_i$ only contains the empty set, we set $\ell_{i+1}=0$ and simply do nothing for all remaining $i$ (these steps can be considered successes).

We repeat this for $I=6\lfloor\log\left(\frac\ell\epsilon\right)\rfloor$ steps. Let \[\mc{U}=\bigcup_{1\le i\le I:\text{step $i$ succeeded}}\mc{C}_i\]Again, letting $i(t)$ to be the highest $i$ such that $\ell_i\ge t$, we have that $\mc{U}_t=\mc{U}_t(H_{i(t)},W_{i(t)})$ and $\mc{U}=\cup_{t=\lceil\ell_{I+1}\rceil}^{\ell(H)}\mc{U}_t$. Now, by our success criteria and our proof of Theorem \ref{kkpp}, we know for sure that \[\sum_{U\in\mc{U}}q^{|U|}\le2\sum_{t=1}^\ell 8^{-t}\binom{\ell_{i(t)}}{t}<\frac 12,\]so as $H$ is $q$-small, $\mc{U}$ for sure does not undercover $H$.

If $\ell_{I+1}<1$, then this means that $\cup_{i=1}^IW_i$ contains a hyperedge in $H$. If we have had at least $\lfloor\log(\ell)\rfloor+1$ successes, then we do have $\ell_{I+1}<1$. We have had $6\lfloor\log\left(\frac\ell\epsilon\right)\rfloor$ steps, each of which had a greater than $\frac 12$ probability of succeeding (regardless of the success or failure of previous steps). Let $X$ be our number of successes, which we then know is stochastically dominated by $Y\sim Bin(I,\frac 12)$. Standard Chernoff bounds give that\[\mb{P}\left(Y\le (1-\delta)\mb{E}Y\right)\le\left(\frac{e^{-\delta}}{(1-\delta)^{1-\delta}}\right)^{\mb{E}Y}.\]
Note that $\frac{\lfloor\log(\ell)\rfloor}{\mb{E}Y}=\frac{\lfloor\log(\ell)\rfloor}{3\lfloor\log(\ell/\epsilon)\rfloor}\le\frac 13$ as $\epsilon\le 1$. So here,\begin{align*}
\mb{P}(\nexists~S\in H\text{ s.t.~}S\subseteq\cup_{i=1}^IW_i)&\le\mb{P}(\ell_{I+1}\ge 1)=\mb{P}(X\le\lfloor\log(\ell)\rfloor)\\&\le\mb{P}(Y\le\lfloor\log(\ell)\rfloor)\le\mb{P}(Y\le(1-2/3)\mb{E}Y)\\&\le\left(\frac{e^{-2/3}}{(1/3)^{(1/3)}}\right)^{\lfloor\log(\ell/\epsilon)\rfloor}<\left(\frac 12\right)^{\log(\ell/\epsilon)-1}=\frac{2\epsilon}{\ell}\le\epsilon
\end{align*}for $\ell\ge 2$. Then $|\cup_{i=1}^IW_i|\le 8Iq|X|\le 48q\log\left(\frac\ell\epsilon\right)|X|$, and once again we can take $W=X_{48q\log(\ell/\epsilon)}$ instead by adding points to $X$ not in any hyperedge in $H$.
\end{proof}
\section{Why Theorem \ref{MainTheorem} is the ``Optimal Convergence Rate''}
Now that we have proven our main theorem, we will explain how it relates to prior work. Park--Pham's paper implied Theorem \ref{MainTheorem} but with a bound of $p=O(q(\log(\ell)+\epsilon^{-c}))$ for some $c\approx 2$ \cite{ParkPham}, an exponentially worse $\epsilon$-dependence than our bound of $p=48q\log\left(\frac{\ell}{\epsilon}\right)$. Our Theorem \ref{MainTheorem} can be rephrased as follows:
\begin{corollary}
Let $H$ be a $\ell$-bounded hypergraph that is not $q$-small and choose any $c\ge 1$. Let $p=48cq\log(2\ell)$. Then $\mb{P}(\exists~S\in H$ s.t.~$S\subseteq X_p)>1-\ell^{-c}$.
\end{corollary}
\begin{proof}
At the end of the proof of Theorem \ref{MainTheorem} our failure probability was $\frac{2\epsilon}{\ell}$. Therefore, we can simply set $\epsilon=\frac 12\ell^{1-c}$, which is at most 1 as required.
\end{proof}
The above corollary is important because it matches the ``with high probability'' statements of prior work, that is, the probability of $X_p$ containing a hyperedge in $H$ goes to 1 as $\ell$ goes to infinity. $\ell^{-c}$ can be thought of as the convergence rate of this probability to 1. The Park--Pham bounds earlier gave for $p=O(q\log(\ell))$ a convergence rate of $(\log(\ell))^{-c}$ for some $c\approx\frac 12$ \cite{ParkPham}, so as before this is an exponential improvement.

One motivation for achieving the bound in Theorem \ref{MainTheorem} was to generalize the similar bound that was achieved under the ``fractional expectation-threshold'' or ``$\kappa$-spread'' framework by Rao \cite{Rao}, improving the $\epsilon$-dependence of previous work \cite{ALWZ,FKNP}.
\begin{definition}[\cite{Talagrand}, Definition 6.6]
$H$ is \underline{$\kappa$-spread} if for all $Y\subseteq X$,\[|\{S\in H:S\subseteq Y\}|\le\kappa^{-|Y|}|H|.\]
\end{definition}
\begin{theorem}[\cite{Rao}, Lemma 4]\label{Rao}
Let $H$ be a $\ell$-bounded hypergraph that is $\kappa$-spread and let $\epsilon\in(0,1)$. There exists a universal constant $\beta$ such that for $p=\frac{\beta}{\kappa}\log\left(\frac{\ell}{\epsilon}\right)$, we have that $\mb{P}(\exists~S\in H$ s.t.~$S\subseteq X_p)>1-\epsilon$.
\end{theorem}
Our Theorem \ref{MainTheorem} is a generalization of Rao's Theorem \ref{Rao} due to the following basic lemma:
\begin{lemma}[\cite{Talagrand}, Propositions 6.2, 6.7]
If $H$ is $\kappa$-spread, $H$ is not $\frac{1}{\kappa}$-small.
\end{lemma}
\begin{proof}
Let $G$ such that $H\subseteq\langle G\rangle$. For all $r\in\mb{N}$, let $n_r=|\{R\in G:|R|=r\}|$ and let $c_r$ be the number of hyperedges $S\in H$ such that $\exists~R\in G$ s.t.~$R\subseteq S$ and $|R|=r$. Because $H$ is $\kappa$-spread, $c_r\le n_r|H|\kappa^{-r}$ for all $r\in\mb{N}$. However, as every hyperedge in $H$ contains some hyperedge in $G$, $\sum_{r\in\mb{N}}c_r\ge|H|\implies\sum_{r\in\mb{N}}\frac{c_r}{|H|}\ge 1$. Putting these inequalities together, $\sum_{R\in G}(\frac{1}{\kappa})^{|R|}=\sum_{r\in\mb{N}}n_r\kappa^{-r}\ge 1>\frac 12$. As this is true for any $G$ such that $H\subseteq\langle G\rangle$, we must have that $H$ is not $\frac{1}{\kappa}$-small.
\end{proof}
Rao's result, and thus ours, is asymptotically optimal (that is, optimal except for the constant) in $\ell$, $\epsilon$, and $\kappa$ (or $q$):
\begin{proposition}[\cite{BCW}, Lemma 4; adapted from \cite{ALWZ}]
Let $\epsilon\in(0,\frac 12]$ and $\kappa,\ell\in\mb{N}$ such that $p=\frac{1}{6\kappa}\log\left(\frac{\ell}{\epsilon}\right)$ has $p\le.7$. There exists a $\kappa$-spread hypergraph $H$ such that $\mb{P}(\exists~S\in H$ s.t.~$S\subseteq X_p)<1-\epsilon$.
\end{proposition}
In conclusion, this paper has successfully generalized the asymptotically optimal bound from the $\kappa$-spread setting to the $q$-small setting, using a different proof technique from those used to achieve this bound in the $\kappa$-spread setting \cite{Rao,Tao,Hu,secmommeth}.
\subsection*{Acknowledgements}
We thank Alan Frieze, Lutz Warnke, and the anonymous referees for their helpful comments.

\end{document}